\documentclass[letterpaper, 10 pt, conference]{ieeeconf}  

\IEEEoverridecommandlockouts                              

\usepackage{adjustbox}
\usepackage{algorithm}
\usepackage{algpseudocode}
\usepackage{amsfonts}
\usepackage{amsmath}
\usepackage{amsthm}
\usepackage{amssymb}
\usepackage{bbm}
\usepackage{bm}
\usepackage{booktabs}
\usepackage{cite}
\usepackage{color, colortbl}
\usepackage{diagbox}
\usepackage[hidelinks]{hyperref}
\usepackage{cleveref}
\usepackage{makecell}
\usepackage{mathrsfs}
\usepackage{mathtools}
\usepackage{multirow, multicol}
\usepackage{slashbox}
\usepackage{threeparttable}
\usepackage[normalem]{ulem}
\usepackage{xcolor}
\usepackage{comment}
\usepackage{graphicx}
\usepackage{subcaption} 
\usepackage{xcolor}
\usepackage{stfloats}

\algrenewcommand\textproc{}
\newcolumntype{M}[1]{>{\centering\arraybackslash}m{#1}}
\newcolumntype{C}[1]{>{\centering\arraybackslash}m{#1}}

\theoremstyle{plain}
\newtheorem{lemma}{\textbf{Lemma}}

\newtheorem{theorem}{\textbf{Theorem}}

\newtheorem{assumption}{Assumption}

\theoremstyle{definition}
\newtheorem{definition}{\textbf{Definition}}
\newtheorem{problem}{\textbf{Problem}}
\newtheorem{remark}{Remark}
\newtheorem{example}{\textbf{Example}}
\useunder{\uline}{\ul}{}

\renewcommand{\baselinestretch}{.91}

\title{\LARGE \bf Reach-Avoid Model Predictive Control with Guaranteed Recursive Feasibility via Input Constrained Backstepping}

\author{Jianqiang Ding$^{1*}$, Nishant Jayesh Bhave$^{2}$ and Shankar A. Deka$^{1*}$,
\IEEEmembership{Member, IEEE}
\thanks{*We acknowledge the financial support of the Finnish Ministry of Education and Culture through the Intelligent Work Machines Doctoral Education Pilot Program (IWM VN/3137/2024-OKM-4).}
\thanks{$^{1}$ Department of Electrical Engineering, Aalto University, Finland. Email:
{\tt\small \{jianqiang.ding, shankar.deka\}}@aalto.fi.}
\thanks{$^{2}$ Department of Electrical Engineering, Indian Institute of Technology Bombay, India. Email:
{\tt\small nishant.bhave}@iitb.ac.in.}
}

\begin{document}

\maketitle
\thispagestyle{empty}
\pagestyle{empty}




\begin{abstract}
    This letter proposes a novel sampled-data model predictive control framework for continuous control-affine nonlinear systems that provides rigorous reach-avoid and recursive feasibility guarantees under physical constraints.
    By propagating both input and output constraints through backstepping process,
    we present a constructive approach to synthesize a reach-avoid invariant set that complies with control input limits.
    Using this reach-avoid set as a terminal set, we prove that the proposed sampled-data MPC framework recursively admits feasible control inputs that safely steer the continuous system into the target set under fast sampling conditions.
    Numerical results demonstrate the efficacy of the proposed approach.
\end{abstract}

\section{Introduction}

Model predictive control (MPC) framework is distinguished by its ability to enforce constraints on system outputs and control inputs, while accommodating a broad class of system dynamics, e.g., \cite{batkovic2023experimental}\cite{kennel2012energy}.
Although the framework has proven effective in mitigating model mismatches and external disturbances by incorporating real-time measurements as control feedback, 
the increasing complexity of modern systems and stricter requirements for operational safety and stability have introduced new challenges.
Motivated by these demands, designing MPC-based controllers with rigorous formal guarantees has emerged as a critical research focus at the intersection of control theory and formal methods.

In the context of safety-critical applications, controller synthesis typically requires that the system avoid undesired states during operation while simultaneously converging to, or remaining within, a target state set.
These two fundamental properties are formally characterized as safety and reachability.
Structurally, the MPC framework achieves guarantee on these properties by enforcing specific state constraints within the optimization problem.
For instance, Control Lyapunov Functions (CLFs) 
are frequently adapted to formulate constraints that guarantee closed-loop stability \cite{minniti2021adaptive}.
Complementarily, safety is typically preserved by incorporating
Control Barrier Functions (CBFs) with MPC frameworks \cite{zeng2021safety}.
However, the mere formulation of such constraints is insufficient to guarantee that the MPC controller can consistently generate valid control signals throughout the system operation. 
A more fundamental prerequisite of MPC is recursive feasibility, which states that if the optimization problem is feasible at the initial time, then it remains feasible at all future time steps. Conventionally, such feasibility is secured by constructing a control invariant set for the terminal states as investigated by \cite{kerrigan2000invariant}.
Nevertheless, constructing such a set is itself a nontrivial task, especially for nonlinear cases, as studied in \cite{bravo2005computation,decardi2021computing,legat2018computing,zhao2024nonlinear}. 

In this paper, we propose a novel MPC framework tailored for control-affine nonlinear systems that provides rigorous reach-avoid guarantees. The main contributions are summarized as below: 
\begin{itemize}
    \item We firstly present an approach that combines Exponential Control Guidance Barrier Functions (ECGBFs) along with backstepping procedure to construct reach-avoid sets for nonlinear systems based on the solution to a simplified optimization problem. 
    \item Specifically, we show how input constraints on the original system can be systematically propagated through the backstepping process, while preserving the structure of the simplified problem.
    
    \item Secondly, we demonstrate how our constructed reach-avoid set maintains its property under additional restrictions on the continuous time controller (zero-order hold). 
    \item This finally allows us to utilize the control-constrained reach-avoid set as the terminal set, along with sampled-data MPC framework to synthesize reach-avoid controllers in a recursively feasible manner. 
\end{itemize}
\textit{Notations:} 
$\mathbb{R}^n$ and $\mathbb{N}$ denote the $n-$dimensional Euclidean space and the set of natural numbers, respectively.
Vectors and matrices are represented by boldface lowercase and uppercase letters (e.g., $\bm{v}$ and $\bm{M}$).
 $\overline{\mathcal{S}}$ denotes closure of a set $\mathcal{S}$. The space of $n-$ times differentiable functions over a domain $\mathcal{S}$ is represented by $\mathcal{C}^n(\mathcal{S})$. $\mathbb{R}[\bm{x}]$ denotes the ring of multivariate polynomials in the variable $\bm{x}$, and $\sum[\bm{x}] \doteq \{ \sum_{i=1}^{k} q_i^2 \, |\, q_i \in \mathbb{R}[\bm{x}], k \in \mathbb{N}
\}$ is the set of sum-of-squares (SOS) polynomials. The $r^{th}$ order Lie derivative of a function $h$ w.r.t vector field $f$ is denoted by $\mathcal{L}_{f}^{r}h$.

\section{Preliminary}
We consider a control-affine nonlinear system of the form
\begin{equation}\label{eq: dynamic system} 
    \frac{d \bm{x}}{dt}
    = \bm{f}(\bm{x}) + \sum_{j=1}^{m}\bm{g}_{j}(\bm{x})u_{j};\quad \bm{y}  = \bm{h}(\bm{x}),
\end{equation}
where the states $\bm{x}\in \mathcal{X}\subseteq \mathbb{R}^{n}$ and control inputs $\bm{u} \doteq [u_{1},\ldots,u_m]^\top$ belong to the set of admissible inputs, $ \mathcal{U} \subseteq \mathbb{R}^m$. The functions $\bm{f}: \mathcal{X}\rightarrow \mathbb{R}^{n}$ and
$\bm{g}_{j}: \mathcal{X}\rightarrow \mathbb{R}^{n},~j\in\{1,\cdots, m\}$ are assumed to be locally Lipschitz, and $\bm{h}
: \mathcal{X} \rightarrow \mathbb{R}^m$ is considered to be a $\mathcal{C}^{\infty} (\mathcal{X})$ function.

We are interested in safe control of the system \eqref{eq: dynamic system}, while guaranteeing convergence to a target set. We formalise the notion of safe set $\mathcal{X}_S$ and target set $\mathcal{X}_T$ using $\mathcal{C}^1$ functions $\psi: \mathbb{R}^m \rightarrow \mathbb{R}$ and $\phi: \mathbb{R}^m \rightarrow \mathbb{R}$ as follows:
\begin{equation}\label{def:safe_target}
    \begin{split}
    \mathcal{X}_S & \doteq \{\bm{x} \in \mathcal{X} \;|\; \ \psi(\bm{y}) > 0,\, \bm{y} = \bm{h}(\bm{x}) \},\\
    \mathcal{X}_T & \doteq \{ \bm{x} \in \mathcal{X} \;|\; \phi (\bm{y}) < 0,\, \bm{y} = \bm{h}(\bm{x}) \}.
\end{split}
\end{equation}
\begin{assumption}\label{assume:safe_target_sets}
    We assume that $\mathcal{X}_T \cap \mathcal{X}_S$ is non-empty, and that $\mathcal{X}_S$ has no isolated points. 
    $\psi(\bm{y})$ is
$\mathcal{C}^1(\mathcal{X}_S)$, and possesses a unique global maximum $\bm{x}^* \in  \mathcal{X}_T \cap \mathcal{X}_S$.
\end{assumption}



\begin{problem} 
    Given a safe set $\mathcal{X}_S$ and a target set $\mathcal{X}_T$ satisfying Assumption \ref{assume:safe_target_sets}, the objective is to synthesize a continuous-time control policy 
    $\bm{u}(t) \in \mathcal{U}$
    for system \ref{eq: dynamic system} such that, from an initial state $\bm{x}_0 \in \mathcal{X}_S \setminus \mathcal{X}_T$, the corresponding closed-loop trajectory enters the target set $\mathcal{X}_T$ at a finite time $\tau := \inf \{ t > 0 \mid x(t) \in \mathcal{X}_T \}$ while 
    remaining safe, i.e., $\bm{x}(\tau)\in \mathcal{X}_T$ and $\bm{x}(t)\in \mathcal{X}_S \;\forall t\in[0,\tau].$ 
    The synthesis of such a control policy 
    over a predictive time interval $[t, t+ T_p]$ can be formally posed as the following optimization problem:
    \begin{equation}\label{eq:continuous_mpc}
    \begin{aligned}
    &\min_{\hat{\bm{u}}(\cdot | t)} \quad  J(\bm{x}(t), \hat{\bm{u}}(\cdot | t)) \\
    &\text{s.t.}  \begin{cases}
    \frac{\partial \hat{\bm{x}}(s | t)}{\partial s} = \bm{f}(\hat{\bm{x}}(s | t)) + \sum_{j=1}^m  \bm{g}_j(\hat{\bm{x}}(s | t))\hat{u}(s | t), \\
    \hat{\bm{x}}(t | t) = \bm{x}(t), \ 
    \hat{\bm{u}}(s | t) \in \mathcal{U}, \
    \hat{\bm{x}}(s | t) \in \mathcal{X}_S, \\
    \hat{\bm{x}}(t+T_p | t) \in \mathcal{X}_T.
    \end{cases}
    \end{aligned}
    \end{equation}
    where $\hat{\bm{x}}(s | t)$ and $\hat{\bm{u}}(s | t)$ denote the predicted state and control signal at future time $s \in [t, t+T_p]$. The cost function $J(\bm{x}(t), \hat{\bm{u}}(\cdot | t))$ 
    is suitably chosen to encourage the model predictive states $\bm{\hat{x}}$ to progressively approach the target set $\mathcal{X}_T$ over the MPC time horizon of $T_p$.
    For example, a candidate cost function can be formulated as
    \begin{equation*}
    J(\bm{x}(t), \hat{\bm{u}}(\cdot | t)) = \text{dist}_{\mathcal{X}_T}(\hat{\bm{x}}(t + T_p | t)) + \int_{t}^{t + T_p} \|\hat{\bm{u}}(s | t)\|^2 ds,
    \end{equation*}
    where $\text{dist}_{\mathcal{X}_T}:\mathcal{X}\rightarrow \mathbb{R}$ denotes the signed distance function from the target set $\mathcal{X}_T$.
\end{problem}

Although such an MPC-based controller can perform reasonably well in specific scenarios, it is generally known not to produce closed-loop behavior with rigorous guarantees. Firstly, the convergence of states to the target set $\mathcal{X}_T$ is not strictly enforced. One could consider a terminal constraint $\hat{\bm{x}}(t+T_p | t) \in \mathcal{X}_T$, potentially requiring a large MPC horizon $T_p$, which can be computationally expensive. Secondly, the closed-loop system with the MPC controller $u^*$ may steer the closed-loop states $\bm{x}(t)$ into parts of the state-space $\mathcal{X}$ from which the MPC problem may no longer be feasible. 
To ensure the system operates safely and eventually entering the given target set $\mathcal{X}_T$, the optimization problem \eqref{eq:continuous_mpc} must satisfy the property of recursive feasibility \cite{katriniok2023discrete,behrunani2024recursive,fang2022model}, which can be formally defined as follows.
\begin{definition} (Recursive Feasibility)
    The MPC optimization problem \eqref{eq:continuous_mpc} is recursively feasible if it is initially feasible at time $t=0$ for a given initial state $\bm{x}_0 \in \mathcal{X}_S \setminus \mathcal{X}_T$, and remains feasible for all subsequent continuous times $t>0$ before entering $\mathcal{X}_T$.
\end{definition}
To address these coupled challenges, we propose to construct a tailored reach-avoid set $\mathcal{X}_{RA}$, defined below, as the terminal set integrated into the MPC optimization problem \eqref{eq:continuous_mpc} to guarantee both the reach-avoid objective and recursive feasibility.

\begin{definition} (Reach-avoid set)
Consider a safe set $\mathcal{X}_S$ and a target set $\mathcal{X}_T$, and the controlled system \eqref{eq: dynamic system}. We call a set $\mathcal{X}_{RA} \subseteq  \mathcal{X}_S$ a reach-avoid set if for every state $\bm{x_0} \in \mathcal{X}_{RA}$, there exists a time $\tau>0$ and control signal $\bm{\hat{u}}(\cdot)$ mapping $[0,\tau] \mapsto \mathbb{R}^m$, such that $\bm{x}(t) \in \mathcal{X}_S$ for all $t \in [0,\tau]$ and $\bm{x}(\tau) \in \mathcal{X}_T$, whenever $\bm{x}(0)=\bm{x_0}.$
\end{definition}

As a foundational step toward synthesizing such terminal set, we first revisit the definition of the Exponential Control Guidance-barrier Function (ECGBF) as presented below.

\begin{definition}
(ECGBF\footnote{We adapt this definition from \cite{xue2024reach} to functions of outputs $\bm{y}$ rather than states $\bm{x}$. One major advantage of doing so is that we are able to employ SOS optimization to compute ECGBFs even for cases when the dynamics, safe set and target set are non-polynomial in $\bm{x}$, for example, robotic manipulators with end-effector constraints.}) Given the safe set $\mathcal{X}_S$ and target set $\mathcal{X}_T$ satisfying assumption \ref{assume:safe_target_sets}, the real-valued function $\psi(\bm{y})$ is an ECGBF if there exists $\lambda>0$ such that
\begin{align} \label{eq: ECGBF constraint}
    \sup_{\bm{u} \in \mathcal{U}} \frac{d}{dt}\psi(\bm{y}) \geq \lambda \psi(\bm{y}),\; \forall\bm{x} \in \overline{\mathcal{X}_S\setminus \mathcal{X}_T}.
\end{align}
\end{definition}

Notably, the zero-superlevel set of the ECGBF inherently satisfies the formal properties of a reach-avoid set. Specifically, it guarantees that for any state $\bm{x}_0$ within this set, there exists a unconstrained control input
$\hat{\bm{u}}(\cdot) \in \mathbb{R}^m$
ensuring the closed-loop system operates safely in $\mathcal{X}_S$ and eventually enters the prescribed target set $\mathcal{X}_T$.
While the standard ECGBF explicitly assumes that the control input $\bm{u}$ must directly influence $\dot{\psi}(\bm{y}(t))$, we do not impose this restrictive requirement in this paper.
To handle multi-input multi-output (MIMO) systems of the form \eqref{eq: dynamic system} that may possess arbitrary vector relative degree, we introduce the foundational result from our previous work \cite{ding2025backstepping} in the following lemma.

\begin{lemma} (Theorem 2 in \cite{ding2025backstepping}) 
    \label{lemma: unconstrained ECGBF for MIMO}
    Given the system \eqref{eq: dynamic system} with safe and target sets satisfying the Assumption \ref{assume:safe_target_sets}, suppose there exists locally Lipschitz continuous function $\bm{k}_1(\bm{y}(\bm{x})) = \begin{bmatrix}
    k_1^1(\bm{y}(\bm{x})) & \cdots & k_1^m(\bm{y}(\bm{x}))
    \end{bmatrix}^\top \in \mathbb{R}^m$ for some $\lambda>0$ such that,
    \begin{equation} \label{eq:ECGBF_single}
    \sum_{i=1}^m \frac{\partial \psi}{\partial y_i} \cdot k_1^i(\bm{y}(\bm{x}))
    \geq \lambda \psi(\bm{y}(\bm{x})), \forall \bm{x} \in \overline{\mathcal{X}_S \setminus \mathcal{X}_T}.
\end{equation}
Let $\{\gamma_1, \cdots, \gamma_m \}$ be the vector relative degree for system \eqref{eq: dynamic system} with the $i-$the output $\bm{y}_i \in \mathbb{R}$ denoted by $\gamma_i$, and $\bm{\eta}^i = [\eta^i_1,\ldots, \eta^i_{\gamma_i}]^\top = [y_{i}(\bm{x}), \ldots, \mathcal{L}_{f}^{r_i-1}y_{i}(\bm{x})]^\top$. Then,
\begin{align}
    \label{eq: RA function for MIMO}
    {\Psi}(\bm{y}(\bm{x})) = \psi(\bm{y}(\bm{x})) - \sum_{i=1}^m \sum_{l=1}^{\gamma_i-1}\frac{1}{2\mu_{l}^i}\|\eta_{l+1}^i-k_{l}^i(\bm{z}_{l}^i)\|_{2}^{2}
\end{align}
is an ECGBF for system \eqref{eq: dynamic system} w.r.t. safe set $\mathcal{X}'_S \doteq \{\bm{x} \in \mathcal{X} \;|\; \ \Psi(\bm{y}(\bm{x})) > 0, \bm{y} = \bm{h}(\bm{x})\}$ and target set $\mathcal{X}_T$,
where $\bm{z}_{l}^i= (\eta_{1}^i, \cdots, \eta_{l}^i)$, $\mu_l^i>0 \ \forall l \in \{1, \cdots, \gamma_i-1\}, i \in \{1, \cdots, m\}$. $k_l^i(\bm{z}_l^i)$ are functions recursively defined as 
\begin{align*}
\bm{k}_2^i(\bm{z}^i_2) &= \mu_{1}^i\frac{\partial \psi(\bm{y}(\bm{x}))}{\partial y_i} + \sum_{s=1}^{1}\frac{\partial k_{1}^i}{\partial \eta_{s}^i}\eta_{s+1}^i + \frac{\lambda}{2}(\eta_{2}^i- k_{1}^i), \\
    \bm{k}_l^i(\bm{z}_l^i) &= - \frac{\mu_{l-1}^i(\eta_{l-1}^i - k_{l-2}^i)}{\mu_{l-2}^i} + \sum_{s=1}^{i-1} \frac{\partial k_{l-1}^i}{\partial \eta_s} \eta_{s+1}^i \\
            & \qquad \quad + \frac{\lambda}{2} (\eta_l^i - k_{l-1}^i), \quad \forall l \in \{3, \cdots, \gamma_i-1\}.
\end{align*}
\end{lemma}
Lemma \ref{lemma: unconstrained ECGBF for MIMO} provides a systematic approach to construct an ECGBF for the MIMO system \eqref{eq: dynamic system} by leveraging backstepping procedure. Its zero-superlevel set $\mathcal{X}_{S}'$ explicitly defines a reach-avoid set $\mathcal{X}_{RA}$ for the system, which is a subset of $\mathcal{X}_S$ since $\Psi(\bm{y}(\bm{x})) \leq \psi(\bm{y}(\bm{x}))$ by construction. 
The key to constructing the reach-avoid set $\mathcal{X}_S'$ fundamentally lies in synthesizing a controller $\bm{k}_1(\bm{y}(\bm{x}))$ that satisfies constraint \eqref{eq:ECGBF_single}. 
This can be achieved via the following SOS program:
\begin{equation}\label{eq:unconstrained k1 SOS}
\begin{aligned}
&\bm{k}_1(\bm{y}(\bm{x})) = \arg\min_{\bm{v}(\bm{y})} \quad \delta \\
\text{s.t.}  & \begin{cases}
[\nabla_{\bm{y}}\psi(\bm{y})]^\top\bm{v} - \lambda \psi + \delta - s_0 \psi - s_1 \phi \in \Sigma [\bm{y}], \\
\delta > 0, \lambda > \epsilon; 
s_0(\bm{y}), s_1(\bm{y}) \in \Sigma[\bm{y}].
\end{cases}
\end{aligned}
\end{equation}

However, the above construction of the reach-avoid set $\mathcal{X}_S'$ inherently assumes unbounded actuation.
Section \ref{subsec:reach-avoid set} extends the above result to rigorously incorporate the given control input constraints $\bm{u} \in \mathcal{U}$, enabling the constructed set to serve as a valid terminal set with the MPC framework to guarantee both the reach-avoid property and recursive feasibility under input constraints.

\section{Recursive Feasible Reach-avoid MPC} \label{sec:method}

This section details of the establishment of the proposed reach-avoid MPC framework.
First, by transforming the control bounds into linear constraints that can be directly integrated into the optimization problem \eqref{eq:unconstrained k1 SOS}, section \ref{subsec:reach-avoid set} formulates an augmented SOS optimization problem to synthesize $\bm{k}_1(\bm{y}(\bm{x}))$ that 
compatible with 
the given physical actuation limits, thereby achieving the constructing of the input-constrained reach-avoid set.
Building upon this foundation, section \ref{subsec:ra_mpc framework} transforms the continuous-time reach-avoid MPC problem \eqref{eq:continuous_mpc} into a practically implementable sampled-data reach-avoid MPC problem. We prove that the sampled-data reach-avoid MPC problem is also recursively feasible under a sufficient small sampling interval.


\subsection{Construction of Input-Constrained Reach-avoid Set} \label{subsec:reach-avoid set}

Beyond providing a constructive approach to define the reach-avoid set $\mathcal{X}_{RA}$, Lemma \ref{lemma: unconstrained ECGBF for MIMO} actually yields, as a fundamental byproduct, a reach-avoid state-feedback controller $\hat{\bm{u}}(\bm{x})$ for the system \eqref{eq: dynamic system}. This $\hat{\bm{u}}(\bm{x})$ can be systematically extracted from $\bm{k}_1(\bm{y}(\bm{x}))$ through output feedback linearization and backstepping techniques \cite{ding2025backstepping}. Under the action of $\hat{\bm{u}}(\bm{x})$, the system state $\bm{x}$ is guaranteed to evolve safely within the given safe set $\mathcal{X}_S$ and ultimately enter the target set $\mathcal{X}_T$. 
Let vector $\bm{b}(\bm{x}) = [b_1(\bm{x}),  \ldots, b_m(\bm{x})]^\top$ with
\begin{gather*}
     b_i(\bm{x}) = - \mathcal{L}_f^{\gamma_i} y_i(\bm{x}) - \frac{\mu_{\gamma_i-1}^i(\eta_{\gamma_i-1}^i- k_{\gamma_i-2}^i)}{\mu_{\gamma-2}} \\
             \qquad \quad \quad + \sum_{s=1}^{\gamma_i-1}\frac{\partial k_{\gamma_i-1}^i}{\partial \eta_{s}^i}\eta_{s+1}^i+ \frac{\lambda}{2}(\eta_{\gamma_i}^i- k_{\gamma_i-1}^i),
\end{gather*}
the reach-avoid controller is given by $\hat{\bm{u}}(\bm{x}) = \bm{A}^{-1} \bm{b}(\bm{x})$ with
\begin{align} \label{matrix A(x)}
\small
\bm{A}(\bm{x})  = \hspace{-0.1cm}
\begin{bmatrix}
\mathcal{L}_{g_1}L_{f}^{r_1 -1 }y_{1}(\bm{x})&\cdots&\mathcal{L}_{g_m}\mathcal{L}_{f}^{r_1 -1 }y_{1}(\bm{x}) \\
\vdots&\ddots&\vdots \\ 
\mathcal{L}_{g_1}\mathcal{L}_{f}^{r_m -1 }y_{m}(\bm{x})&\cdots&\mathcal{L}_{g_m}\mathcal{L}_{f}^{r_m -1 }y_{m}(\bm{x})
\end{bmatrix} \hspace{-0.1cm}.
\end{align}

We observe that this backstepping framework establishes an explicit mapping from $\bm{k}_1(\bm{y}(\bm{x}))$ controller to the actual control input $\hat{\bm{u}}(\bm{x})$, allowing us to backpropagate the hard constraints imposed on the actual control input $\hat{\bm{u}}(\bm{x})$ back into synthesizing of $\bm{k}_1(\bm{y}(\bm{x}))$. 
We summarize this critical theoretical extension in the following theorem.

\begin{theorem} \label{theorem:affine structure}
    Suppose each component of $\bm{k}_1(\bm{y}(\bm{x}))$ is linearly parameterized as $\bm{k}^i_1(\bm{y}(\bm{x})) = \bm{\theta}_i^\top \mathcal{B}(\bm{y})$ for all $i \in \{1, \cdots, m \}$, where $\bm{\theta}_i$ is a coefficient vector and $\mathcal{B}(\bm{y})$ is a vector of basis functions. Then, any reach-avoid controller $\hat{\bm{u}}(\bm{x})$ induced by Lemma \ref{lemma: unconstrained ECGBF for MIMO} is affine w.r.t. the aggregated coefficient matrix $\bm{\Theta} = 
    \begin{bmatrix}
        \bm{\theta}_1^\top &
        \cdots & \bm{\theta}_m^\top
    \end{bmatrix}^\top$.
\end{theorem}

\begin{proof}
    By assumption, each $\bm{k}^i_1(\bm{y}(\bm{x}))$ takes the affine form $\bm{k}_1^i(z_1^i)  = \bm{\theta}_i^\top \mathcal{G}_1^i(y(x)) + \mathcal{H}_1^i(y(x))$ with $\mathcal{G}_1^i=\mathcal{B}$ and $\mathcal{H}_1^i=0$.
    Throughout the recursive backstepping formulation, this affine structure is preserved under linear operations and partial differentiations. By induction assuming that 
    at step $l \in \{3, \cdots, \gamma_i -1\}$, the previous two controllers $k_{l-2}^i(z_{l-2}^i)$ and $k_{l-1}^i(z_{l-1}^i)$ are both affine in $\bm{\theta}_i$ such that $\bm{k}_j^i(\bm{z}_j^i) = \bm{\theta}_i^\top \mathcal{G}_j^i(x) + \mathcal{H}_j^i(x)$ for $j \in \{l-1,l-2\}$,    
    the subsequent controller $k_l^i(z_l^i)$ yields $\bm{k}_l^i(\bm{z}_l^i)= \bm{\theta}_i^\top \mathcal{G}_l^i(x) + \mathcal{H}_l^i(x)$
    with 
    \begin{align*}
        \mathcal{G}_l^i(x) &= \frac{\mu_{l-1}^i}{\mu_{l-2}^i} \mathcal{G}_{l-2}^i(x) 
        +\sum_{s=1}^{l=1} \frac{\partial \mathcal{G}_{l-1}^i(x)}{\partial \eta_s} \eta_{s+1}^i
        - \frac{\lambda}{2} \mathcal{G}_{l-1}^i(x) \\
        \mathcal{H}_l^i(x) &= - \frac{\mu_{l-1}^i [\eta_{l-1}^i - \mathcal{H}_{l-2}(x)]}{\mu_{l-2}^i}  
        + \sum_{s=1}^{l=1} \frac{\partial \mathcal{H}_{l-1}^i(x)}{\partial \eta_s} \eta_{s+1}^i \\
        & \qquad \qquad \qquad \qquad \ \quad \qquad + \frac{\lambda [\eta_l^i - \mathcal{H}_{l-1}^i(x)]}{2}.
    \end{align*}
    Then, we have
    $
        b_i(x) 
        = \bm{\theta}_i^\top \mathcal{U}_i(x) + \mathcal{V}_i(x) =  \mathcal{U}_i^\top(x) \bm{\theta}_i + \mathcal{V}_i(x)$
    where $\mathcal{U}_i(x)$ and $\mathcal{V}_i(x)$ aggregate the respective recursive elements alongside the unactuated dynamics $- \mathcal{L}_f^{\gamma_i} h_i(\bm{x})$.
    Given $\bm{\Theta} = 
        \begin{bmatrix}
            \bm{\theta}_1^\top &
            \cdots & \bm{\theta}_m^\top
        \end{bmatrix}^\top \in \mathbb{R}^{\sum_{s=1}^m |\bm{\theta}_i| \times 1}$, the actual controller $\bm{u}(x)$ can be written as,
    \begin{align*}
        \hat{\bm{u}}(x) = \bm{A}^{-1}(x) \bm{b}(x)  
        = \bm{A}^{-1}(x)  \bm{\mathcal{U}}(x) \bm{\Theta} + \bm{A}^{-1}(x) \bm{\mathcal{V}}(x),
    \end{align*}
    with $\bm{\mathcal{U}}(x) = \text{diag}[\mathcal{U}^\top_1(x), \cdots, \mathcal{U}^\top_m(x)]$ and $\bm{\mathcal{V}}(x) = [\mathcal{V}_1(x), \cdots, \mathcal{V}_m(x)]^\top$,
    which completes the proof.
\end{proof}

Theorem \ref{theorem:affine structure} implies that when fixing all $\mu_l^i$ and $\lambda$ as chosen constants, the actual control input $\hat{\bm{u}}(\bm{x})$ remains strictly affine in the coefficient vector $\bm{\Theta}$. 
Therefore, any constraint on $\hat{\bm{u}}(\bm{x})$ that preserves this affine structure can be directly translated into linear constraints in the SOS program \eqref{eq:unconstrained k1 SOS} for solving the corresponding input-constrained $\bm{k}_1(\bm{y}(\bm{x}))$ controller.
Specifically, we consider the control input constraints taking the form,
\begin{align} \label{eq:control constraints}
    \mathcal{\bm{P}}(\bm{x}) \bm{u}(\bm{x}) \leq \mathcal{\bm{Q}}(\bm{x}) \quad \forall \bm{x} \in  \overline{\mathcal{X}_S \setminus \mathcal{X}_T}, 
\end{align}
where $\mathcal{P}(\bm{x})$ and $\mathcal{Q}(\bm{x})$ represent a state-dependent polynomial matrix and vector, respectively.
By incorporating such bounds into the problem \eqref{eq:unconstrained k1 SOS}, we formulate an augmented SOS optimization problem as below to synthesize a valid controller $\hat{\bm{k}}_1(\bm{y}(\bm{x}))$ which ultimately results in a reach-avoid controller $\hat{\bm{u}}(\bm{x})$ that satisfies the contraint \eqref{eq:control constraints}:
\begin{equation}\label{eq:constrained_k1_SOS}
\hspace{-0.45cm}
\begin{aligned}
\hat{\bm{\Theta}} &= \arg\min_{\bm{\Theta}, \delta, \lambda} \quad \delta \\
\text{s.t.} & \begin{cases}
[\nabla_{\bm{y}}\psi(\bm{y})]^\top \bm{\mathcal{B}} \bm{\Theta} - \lambda \psi + \delta - s_0 \psi - s_1 \phi \in \Sigma [\bm{y}], \\
\mathcal{\bm{P}}\text{Adj}(\bm{A}) \big(   \bm{\mathcal{U}} \bm{\Theta} + \bm{\mathcal{V}} \big) - \text{det}(\bm{A})\mathcal{\bm{Q}} - S_2 \psi(h(\bm{x}))\\ - S_3 \phi(h(\bm{x})) \in \Sigma [\bm{x}], \\
\delta > 0,\lambda > \epsilon; 
s_0(\bm{y}), s_1(\bm{y}) \in \Sigma[\bm{y}],\\ S_2(\bm{x}), S_3(\bm{x}) \in \Sigma[\bm{x}],
\end{cases}
\end{aligned}
\hspace{-1cm}
\end{equation}
where $S_2, S_3$ are SOS polynomial vectors of length $m$.



\begin{remark}
    While Theorem \ref{theorem:affine structure} enables enforcing constraints on the reach-avoid controller via the SOS program \eqref{eq:constrained_k1_SOS}, solving such a formulation directly inevitably incurs prohibitive computational costs if the state dimensions are large. To circumvent this bottleneck, we employ a sampling-based relaxation strategy that converts the second constraint of program \eqref{eq:constrained_k1_SOS} into a linear constraint (as opposed to a semi-definite program). This implementational detail can be found in \href{https://github.com/Aalto-Nonlinear-Systems-and-Control/reach_avoid_backstepping_MPC.git}{https://github.com/Aalto-Nonlinear-Systems-and-Control/reach\_avoid\_backstepping\_MPC.git}
\end{remark}

\subsection{Reach-avoid MPC with Recursive Feasibility Guarantees} \label{subsec:ra_mpc framework}

In practice, directly solving the continuous-time MPC problem \eqref{eq:continuous_mpc} is intractable due to infinite-dimensional decision space and continuous state constraints.
We thus transition to a tractable sampled-data MPC paradigm by parameterizing the control input as piecewise constant over sampling intervals.
Before formulating the equivalent MPC framework, we 
adapt the methodology from \cite{wu2008semi} to show 
that the continuous-time reach-avoid controller $\bm{u}(\bm{x})$ 
preserves the reach-avoid guarantee under a zero-order hold (ZOH) implementation given appropriate sampling conditions. We formalize this result in the following lemma.

\begin{lemma} \label{lemma:sampled-data system reach-avoid guarantees}
    Suppose the system \eqref{eq: dynamic system} with safe set $\mathcal{X}_S$ and target set $\mathcal{X}_T$ satisfying Assumption \ref{assume:safe_target_sets}, and actuated by a ZOH implementation of the continuous-time reach-avoid controller, $\bm{u}^c (\bm{x})$, such that for any $t \in [t_k, t_k+T)$, the resulting closed-loop dynamic are given by:
    \begin{align} \label{eq:zoh closed-loop system}
        \frac{d \bm{x}}{dt}
        = \bm{f}(\bm{x}(t)) + \sum_{j=1}^{m}\bm{g}_{j}(\bm{x}(t))u_{j}^c(t_k); \ \bm{y}  = \bm{h}(\bm{x}(t)).
    \end{align}
    Then 
    there exists a sampling period upper bound $T^* >0$ such that for any $T \in (0, T^*)$, all trajectories of the closed-loop system \eqref{eq:zoh closed-loop system} originating from the corresponding reach-avoid set $\mathcal{X}_{RA}$ remain within $\mathcal{X}_S$ before eventually entering $\mathcal{X}_T$.
\end{lemma}

\begin{proof}
    Let $\dot{\bm{\Psi}}_c(t)$ and $\dot{\bm{\Psi}}_s(t)$ denote the time derivative of $\bm{\Psi}(\bm{y})$ in \eqref{eq: RA function for MIMO} driven by the continuous-time reach-avoid controller $\bm{u}^c(\bm{x})$ and its ZOH implementation such that:
    \begin{align*}
        \dot{\bm{\Psi}}_c(t)
        &=
        \frac{\partial \Psi(\bm{x}(t))}{\partial \bm{x}} 
        \bigg[
        \bm{f}(\bm{x}(t)) + \sum_{j=1}^{m}\bm{g}_{j}(\bm{x}(t))u_{j}^c(t)
        \bigg].
    \end{align*}
    By applying the Cauchy-Schwarz inequality, we have:
    \begin{align*}
        \dot{\bm{\Psi}}_s(t)
        &=
        \dot{\bm{\Psi}}_c(t) + 
        \sum_{j=1}^{m} \frac{\partial \Psi(\bm{x}(t))}{\partial \bm{x}} \bm{g}_{j}(\bm{x}(t))
        \bigg[ u_{j}^c(t_k) - u_{j}^c(t) \bigg] \\
        &\geq
        \dot{\bm{\Psi}}_c(t) - 
        \sum_{j=1}^m M_{\bm{g}_j} M_{\bm{u}_j^c} \left\| \bm{x}(t) - \bm{x}(t_k) \right\|
    \end{align*}
    where positive constants $M_{\bm{u}_j^c}$ and $M_{\bm{g}_j}$ strictly bound the magnitudes $|\bm{u}_j^c(\bm{x}) |$ and $\left\| \frac{\partial \Psi(\bm{x}(t))}{\partial \bm{x}} \bm{g}_{j}(\bm{x}(t)) \right\|$ respectively over the compact domain $\mathcal{X}$.
    Integrating the sampled-data dynamics from $t_k$ to $t$, we have
    \begin{align*}
        \bm{x}(t) - \bm{x}(t_k) = 
        \int_{t_k}^t 
        \bigg[\bm{f}(\bm{x}(s)) + \sum_{j=1}^{m}\bm{g}_{j}(\bm{x}(s))u_{j}^c(t_k) \bigg] \mathrm{d}s.
    \end{align*}
    Let $\ell_{\bm{f}}$ and $\ell_{\bm{g}_j}$ be the Lipchitz constants of $\bm{f}(\bm{x})$ and each corresponding $\bm{g}_j(\bm{x})$,
    by taking the norm of the integrand, then apply the triangle inequality and the Grönwall–Bellman inequality, the state drift 
    can be bounded as,
    \begin{align*}
        \left\| \bm{x}(t) - \bm{x}(t_k) \right\| 
        &\leq 
        \int_{t_k}^t
        \bigg[
        \ell_{\bm{f}}
        \left\| \bm{x}(s) - \bm{x}(t_k) \right\| + \left\| \dot{\bm{x}}(t_k) \right\|\\
        & \qquad \quad + \sum_{j=1}^m \ell_{\bm{g}_j} M_{u_j^c} \left\| \bm{x}(s) - \bm{x}(t_k) \right\|
        \bigg]
        \mathrm{d}s \\
        &\leq \frac{e^{\left(\ell_{\bm{f}} + \sum_{j=1}^m \ell_{\bm{g}_j} M_{u_j^c}\right)T} -1}{\ell_{\bm{f}} + \sum_{j=1}^m \ell_{\bm{g}_j} M_{u_j^c}} \left\| \dot{\bm{x}}(t_k) \right\| .
    \end{align*}
    Then we have,
    \begin{align*}
        \dot{\Psi}_s(\bm{x}(t)) &\geq \dot{\Psi}_c(\bm{x}(t))  - 
        \sum_{j=1}^m M_{\bm{g}_j} \ell_{\bm{u}_j} \left\| \bm{x}(t) - \bm{x}(t_k) \right\| \\
        &\geq \lambda \Psi_c(\bm{x}(t)) - \beta(T), \quad \forall t \in [t_k, t_k+T],
    \end{align*}
    where $\beta(T)= \left( \sum_{j=1}^m M_{\bm{g}_j} M_{\bm{u}_j} \right) \frac{e^{\ell_{sys}T} -1}{\ell_{sys}} \left\| \dot{\bm{x}}(t_k) \right\|$ and 
    $\ell_{sys} = \ell_{\bm{f}} + \sum_{j=1}^m \ell_{\bm{g}_j} M_{u_j^c}$. Using comparison lemma \cite{khalil2002nonlinear} over $t \in [t_k, t_k+T)$ we have,
    \begin{align*}
    \Psi_s(\bm{x}(t)) \geq e^{\lambda(t-t_k)} \Psi_c(\bm{x}(t_k)) + \frac{\beta(T)}{\lambda} \left( 1 - e^{\lambda(t-t_k)} \right).
    \end{align*}
    Let $T \in (0, T^*)$ such that $\Psi_c(\bm{x}(t_k)) - \frac{\beta(T^*)}{\lambda}> 0$. We obtain
    \small{
    \begin{align*}
        \Psi_s(\bm{x}(t_{k+1}))  - \Psi_s(\bm{x}(t_k)) &\geq (e^{\lambda T}  - 1 ) \left[ \Psi_c(\bm{x}(t_k)) - \frac{\beta(T)}{\lambda} \right] \\
        &> 0 , \quad \forall \bm{x} \in \overline{\mathcal{X}_{RA} \setminus \mathcal{X}_T}.
    \end{align*}}
    Suppose 
    a sampled-data trajectory with sampling period $T \in (0, T^*)$ originates in $\mathcal{X}_{RA}$ but never enters $\mathcal{X}_T$.
    The trajectory would then remain permanently within the compact set $\overline{\mathcal{X}_{RA} \setminus \mathcal{X}_T}$ where $\Psi(\bm{x})$ strictly increases.
    This strict monotonic increase
    implies $\lim_{k \rightarrow \infty} \Psi(\bm{x}(t_k)) \rightarrow +\infty$ over infinite sampling steps
    which contradicts
    the boundedness of the continuous function $\Psi(\bm{x})$ on the compact set $\mathcal{X}_{RA}$. Therefore, 
    all sampled-data trajectories 
    must remain within $\mathcal{X}_{RA} \subseteq \mathcal{X}_S$ before eventually reaching $\mathcal{X}_T$.
\end{proof}
\captionsetup[subfigure]{justification=centering,singlelinecheck=false}
\begin{figure*}[t]
    \centering
    \begin{subfigure}[t]{0.23\linewidth}
        \centering
        \includegraphics[width=\linewidth]{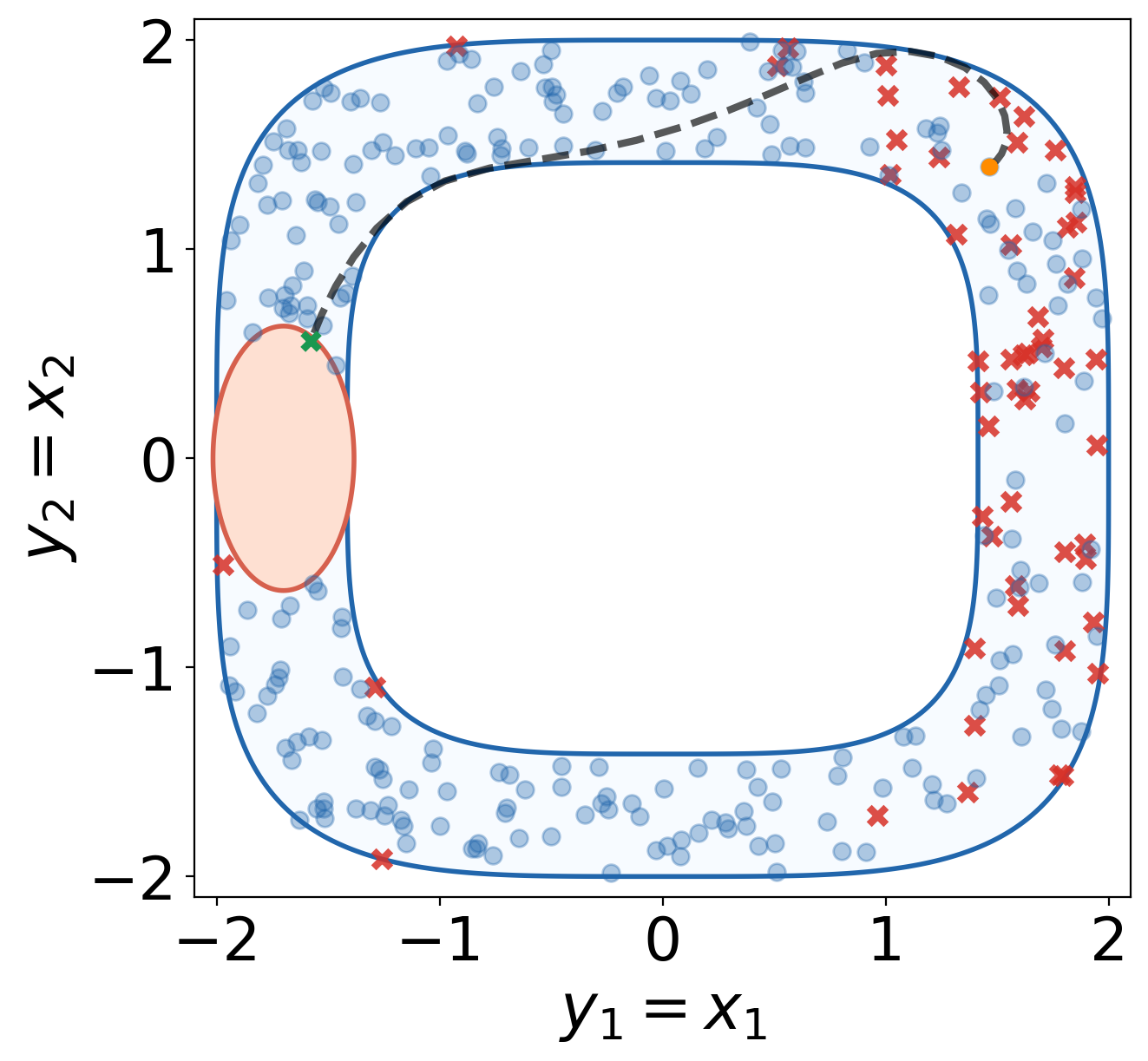}
        \caption{Vanilla MPC with 245/300.}
        \label{fig:vanilla mpc}
    \end{subfigure}
    \hfill
    \begin{subfigure}[t]{0.23\linewidth}
        \centering
        \includegraphics[width=\linewidth]{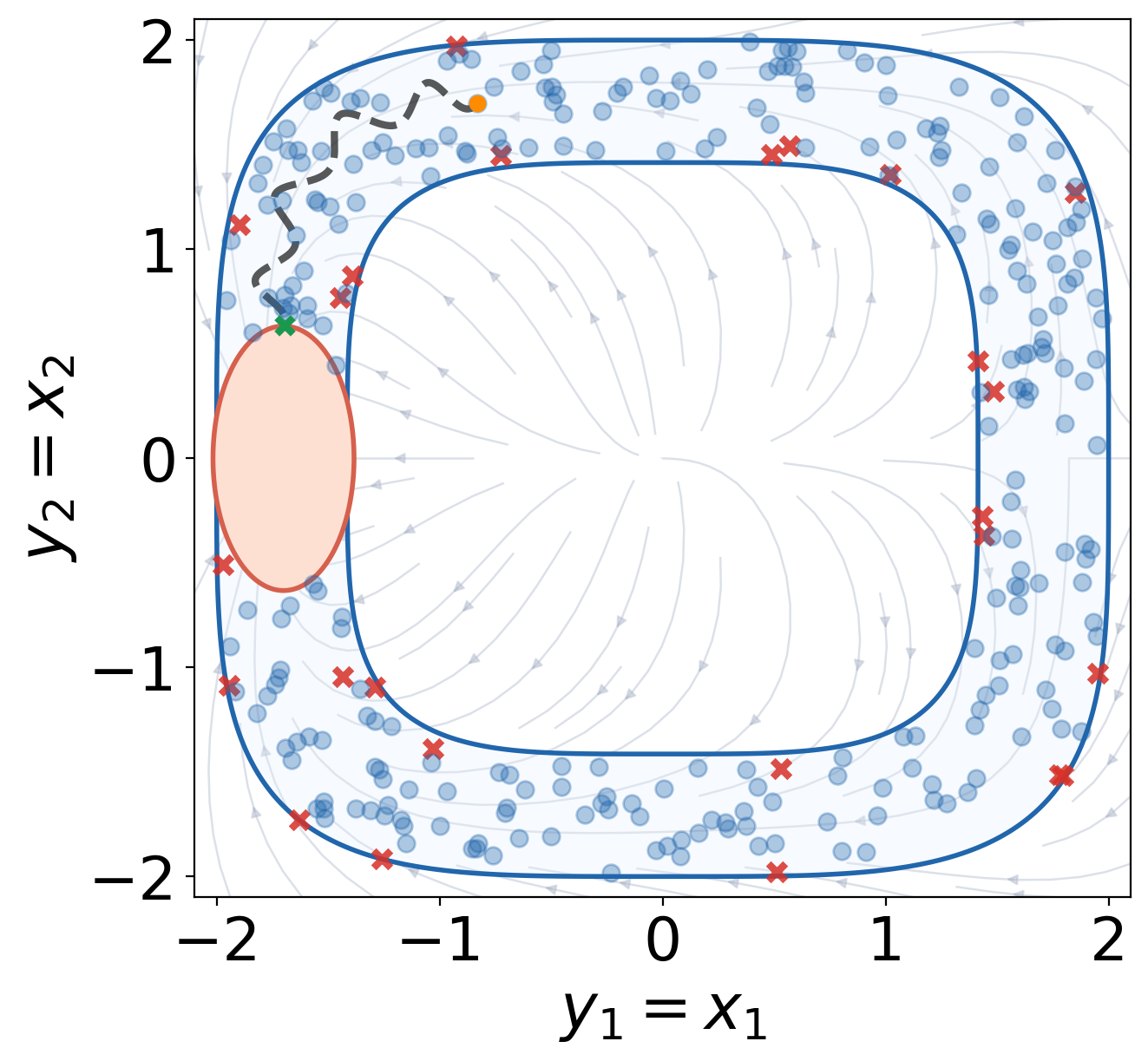}
        \caption{Unconstrained reach-avoid with  275/300.}
        \label{fig:unconstrained reach-avoid}
    \end{subfigure}
    \hfill
    \begin{subfigure}[t]{0.23\linewidth}
        \centering
        \includegraphics[width=\linewidth]{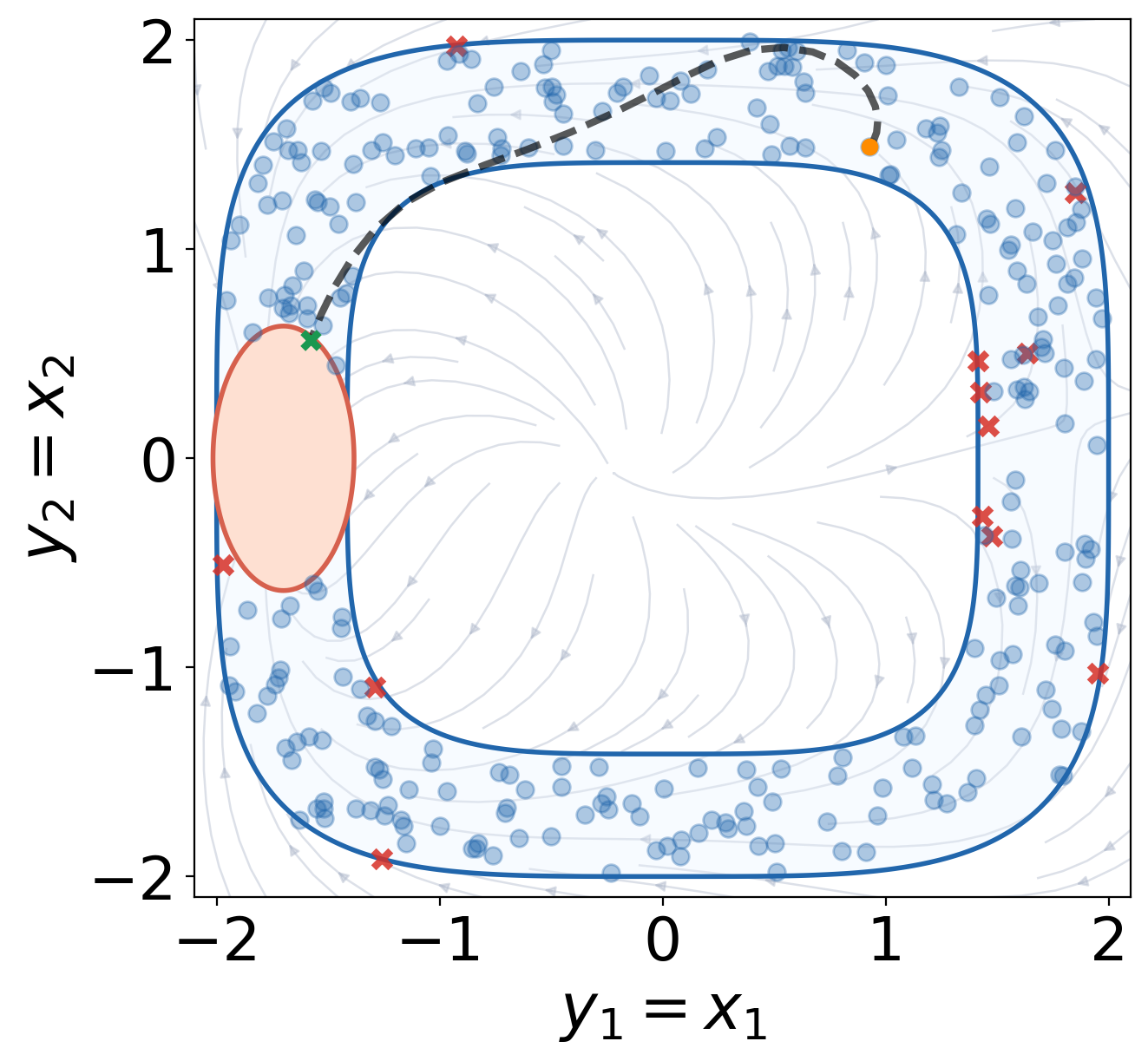}
        \caption{Our method with 288/300.}
        \label{fig:constrained reach-avoid mpc}
    \end{subfigure}
    \hfill
    \begin{subfigure}[t]{0.29\linewidth}
        \centering
        \includegraphics[width=\linewidth]{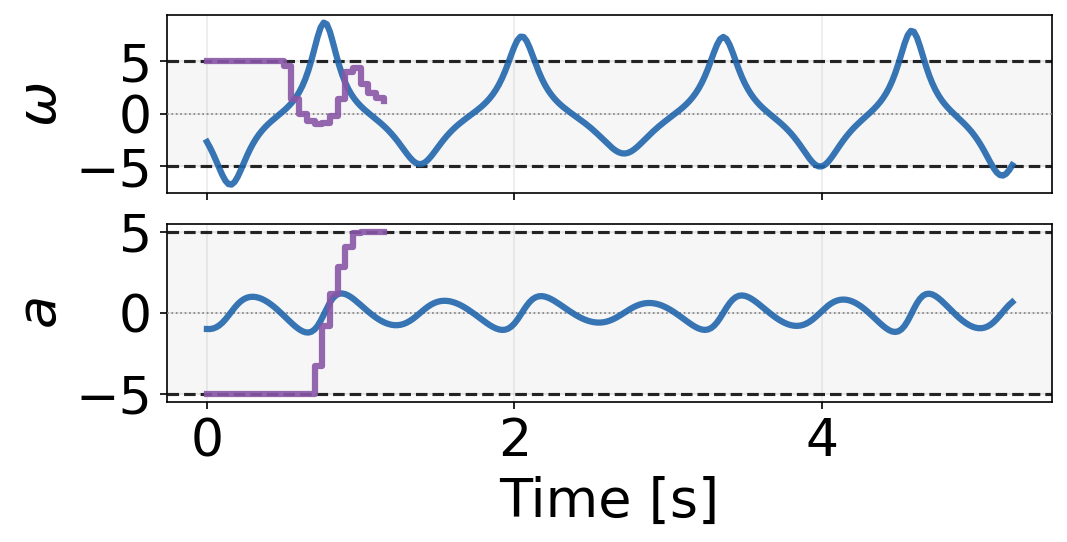}
        \caption{Control inputs comparison}
        \label{fig:control inputs comparison}
    \end{subfigure}    
    \caption{
    (\subref{fig:vanilla mpc}), (\subref{fig:unconstrained reach-avoid}) \& (\subref{fig:constrained reach-avoid mpc}) show controller success rates for $300$ random initial states. Blue dots indicate successful states achieving the reach-avoid objective, while red crosses represent failed ones. A closed-loop trajectory (dashed line from orange dot to green cross) is included for visualization. Initial angles are omitted for clarity. (\subref{fig:control inputs comparison}) represents control inputs comparison between unconstrained reach-avoid controller (blue) and the proposed MPC controller (purple).
    }
    \label{fig:ex10}\vspace{-0.5cm}
\end{figure*}

Lemma \ref{lemma:sampled-data system reach-avoid guarantees} demonstrates that sufficiently fast sampling preserves the reach-avoid \footnote{Note that the exponential convergence rate is generally lost under sampled-data execution as illustrated in the proof of Lemma \ref{lemma:sampled-data system reach-avoid guarantees}.} guarantees for sampled-data control. 
Building upon this theoretical foundation, the following theorem establishes the complete sampled-data reach-avoid MPC framework with guaranteed recursive feasibility.

\begin{theorem}  
    Given the continuous-time dynamical system \eqref{eq: dynamic system} with safe set $\mathcal{X}_S$ and target set $\mathcal{X}_T$ both satisfying Assumption \ref{assume:safe_target_sets},
    suppose the sampling condition established in Lemma \ref{lemma:sampled-data system reach-avoid guarantees} holds. If the sampled-data reach-avoid MPC problem formulated below is feasible at the initial step $k=0$, then there exists a sampled-data control policy that rigorously guarantees recursive feasibility for all subsequent steps $k \in \mathbb{N}$ and successfully achieves the reach-avoid objective for the corresponding closed-loop system:
    \begin{equation} \label{eq:sampled-data MPC} \hspace{-0.2cm}
        \begin{aligned}
        \hspace{0cm} &\min_{\mathbf{U}_k} \quad  J(\bm{x}_k, \mathbf{U}_k) \\
        \hspace{0cm} &\text{s.t.}  \begin{cases}
        \frac{\partial \hat{\bm{x}}(\tau | k)}{\partial \tau} = \bm{f}(\hat{\bm{x}}(\tau | k)) + \sum_{j=1}^m  \bm{g}_j(\hat{\bm{x}}(\tau | k))\bm{u}_{i|k}, \\ 
        \qquad \qquad \qquad \qquad \qquad  \forall \tau \in [iT, (i+1)T); \\
        \hat{\bm{x}}(0 | k) = \bm{x}_k; 
        \bm{u}_{i|k} \in \mathcal{U}, \ \forall i \in \{0, \dots, N-1\}; \\
        \hat{\bm{x}}(\tau | k) \in \mathcal{X}_S, \ \forall \tau \in [0, NT]; \
        \hat{\bm{x}}(NT | k) \in \mathcal{X}_{RA};
        \end{cases}
        \end{aligned}
        \hspace{-0.5cm}
    \end{equation}
    where $\bm{x}_k = \bm{x}(kT)$ is the sampled state at step $k$ and $\mathbf{U}_k = \{ \bm{u}_{0|k} \cdots, \bm{u}_{N-1 | k} \}$ represents the optimal control sequence over the prediction horizon $N$. $\tau \in [0, NT]$ denotes the relative continuous prediction time instant, and $\hat{\bm{x}}(\tau | k)$ defines the predicted continuous state trajectory driven by the piecewise constant control input.
\end{theorem}

\begin{proof}
    By assumption, the sampled-data MPC problem \eqref{eq:sampled-data MPC} is feasible at the initial step $k=0$. 
    Let $\mathcal{X}_{RA}$ denotes the reach-avoid set synthesized from a valid $\bm{k}_1(\bm{y}(\bm{x}))$ controller of the SOS problem \eqref{eq:constrained_k1_SOS}, and $\bm{u}(\bm{x})$ represents the corresponding continuous-time reach-avoid controller extracted from $\bm{k}_1(\bm{y}(\bm{x}))$ via backstepping.
    Suppose an optimal control sequence $\mathbf{U}_k^* = \{ \bm{u}_{0|k}^* \cdots, \bm{u}_{N-1 | k}^* \}$ exists at step $k$. We define the candidate control sequence for step $k+1$ as $\mathbf{U}_{k+1} = \{\bm{u}_{1|k}^*, \cdots, \bm{u}_{N-1|k}^*, \bm{u}(\hat{\bm{x}}^*(NT|k)) \}$.
    Since the first $N-1$ control inputs of $\mathbf{U}_{k+1}$ are inherited from $\mathbf{U}_k^*$, they naturally satisfy the required reach-avoid properties by assumption. Given that $\bm{u}(\hat{\bm{x}}^*(NT|k))$ is essentially a sampled control input of the continuous-time reach-avoid controller $\bm{u}(\bm{x})$ at time time $NT|k$. 
    According to Lemma \ref{lemma:sampled-data system reach-avoid guarantees}, given a sufficiently small sampling period $T$, the closed-loop trajectory under the control of $\bm{u}(\hat{\bm{x}}^*(NT|k))$ over the prediction interval $\tau \in [(k+N)T, (k+N+1)T]$ remains strictly within the safe set $\mathcal{X}_S$ and the terminal state $\bm{x}(NT|k+1)$ stays inside $\mathcal{X}_{RA}$.
    Therefore the candidate control sequence $\mathbf{U}_{k+1}$ constitutes a feasible solution at step $k+1$. By induction this establishes strict recursive feasibility for the proposed sampled-data reach-avoid MPC framework across all discrete steps $k \in \mathbb{N}^+$, which completes the proof.
\end{proof}

\section{Experiments}

In this section, we demonstrate the proposed methodology on two under actuated dynamical systems. We use SOSTOOLS \cite{sostools} toolbox with Mosek \cite{aps2019mosek} solver to formulate and solve SOS optimization problems.
The complete source code and simulation configurations are available at \href{https://github.com/Aalto-Nonlinear-Systems-and-Control/reach_avoid_backstepping_MPC.git}{https://github.com/Aalto-Nonlinear-Systems-and-Control/reach\_avoid\_backstepping\_MPC.git}


\begin{example}

Consider the modified Dubins car model
\begin{equation}
    \dot{\bm{x}} = 
    \left[\begin{array}{c}\dot{x}_1\\\dot{x}_2\\\dot{\theta}\\\dot{v}\end{array}\right]= 
    \left[\begin{array}{c}v\cos{\theta}\\v\sin{\theta}\\0\\0\end{array}\right] + \left[\begin{array}{cc}0 & 0\\0 & 0\\1 & 0\\0&1\end{array}\right]
    \left[\begin{array}{c}\omega\\a\end{array}\right],
    \label{eq.dubins_car}
\end{equation}
with $\bm{x} = [x_1,x_2,\theta,v]^\top \in \mathbb{R}^4$, where $(x_1,x_2)$ denotes the planar location, $\theta$ represents the heading angle, and $v$ is the forward velocity. The vehicle is actuated by the angular velocity $\omega$ and forward acceleration $a$ as control inputs $\bm{u} = [\omega, a]^\top \in \mathbb{R}^2$. 
Given the system output
$\bm{y} = [y_1, y_2]^\top = [x_1, x_2]^\top$, the control objective is to safely steer the initial states within the safe set $\mathcal{X}_S = \{ \bm{x} \in \mathbb{R}^4 \mid \psi(\bm{y}) > 0 \}$ into the target set $\mathcal{X}_T = \{ \bm{x} \in \mathbb{R}^4 \mid \phi(\bm{y}) < 0 \}$ where $\phi(\bm{y}) = y_1^2+4(y_2+1.7)^2 - 0.4$ and $\psi(\bm{y}) = 10^{-3} (\phi(\bm{y})-300) (y_1^4 + y_2^4 -16) (y_1^4 + y_2^4 -4).$
We compare three distinct methods namely vanilla MPC, the unconstrained reach-avoid controller and the proposed constrained reach-avoid MPC to accomplish this task on $300$ safe initial states randomly sampled outside the target set. 
Both vanilla MPC and the proposed framework minimize quadratic cost under actuator limits $|u_{1,2}| \leq 5$ with prediction horizon $N=25$ and sampling step $T=0.05s$. 
Unlike vanilla MPC which utilizes the target set as terminal constraint, the proposed framework leverages the computed constrained reach-avoid set.
As shown in Fig. \ref{fig:constrained reach-avoid mpc}, the proposed MPC achieves a significantly broader feasible domain than vanilla MPC in Fig. \ref{fig:vanilla mpc} under identical prediction horizons due to the computed $\mathcal{X}_{RA}$ as terminal set.
Furthermore, compared to the unconstrained reach-avoid controller in Fig. \ref{fig:unconstrained reach-avoid}, our MPC optimization yields smoother trajectories with minimal control energy.
Fig. \ref{fig:control inputs comparison} demonstrates that while the unconstrained reach-avoid controller guarantees the reach-avoid property, it severely violates physical actuator bounds and requires more time to reach the target despite starting closer. 
Conversely the proposed reach-avoid MPC reach the target faster with less control effort while maintaining all control inputs strictly within the limits throughout the entire evolution.

\end{example}

    \captionsetup[subfigure]{justification=centering,singlelinecheck=false}
\begin{figure*}[t]
    \centering
    \begin{subfigure}[t]{0.23\linewidth}
        \centering
        \includegraphics[width=\linewidth]{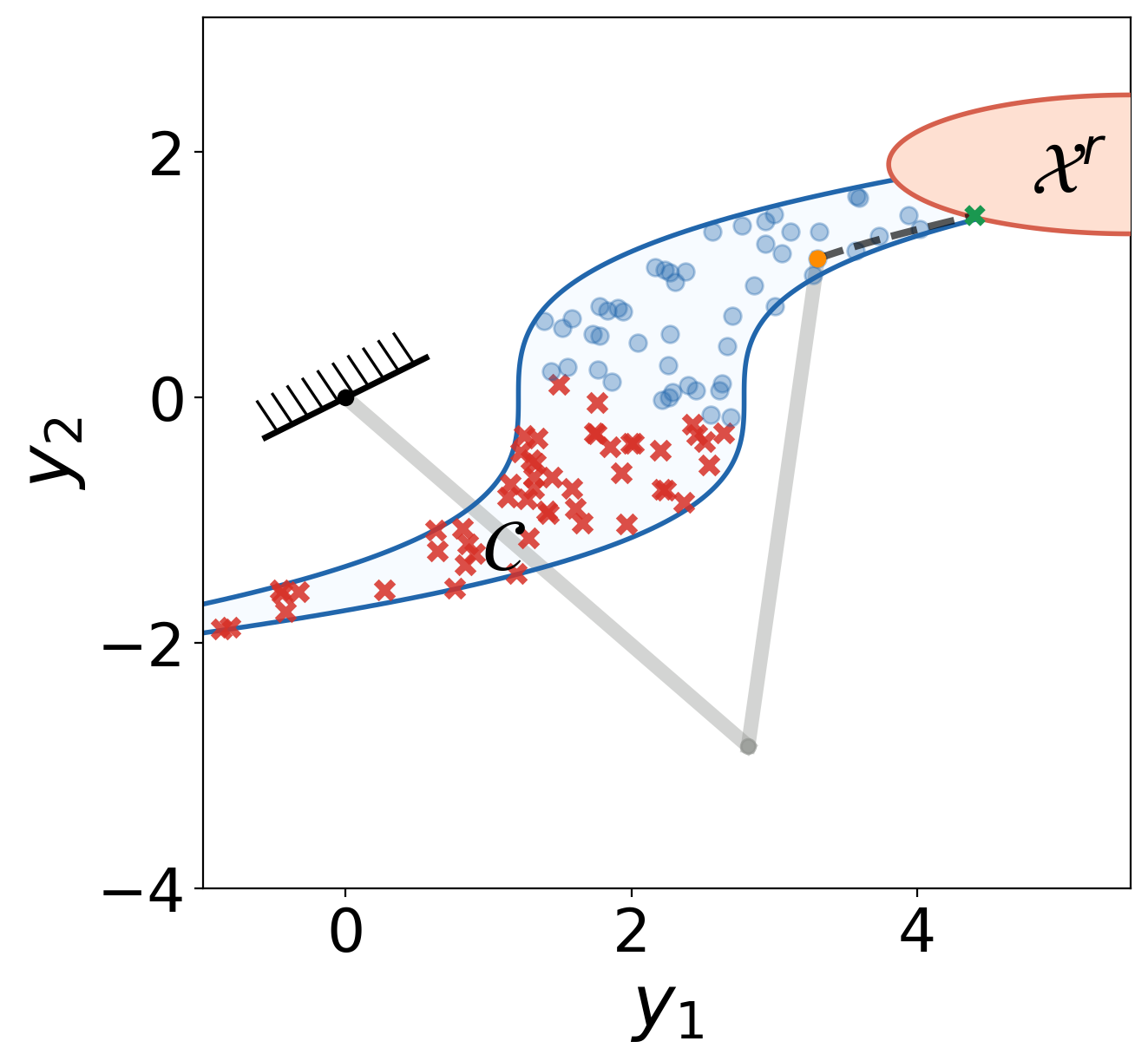}
        \caption{Vanilla MPC with 50/100.}
        \label{fig:vanilla mpc manipulator}
    \end{subfigure}
    \hfill
    \begin{subfigure}[t]{0.23\linewidth}
        \centering
        \includegraphics[width=\linewidth]{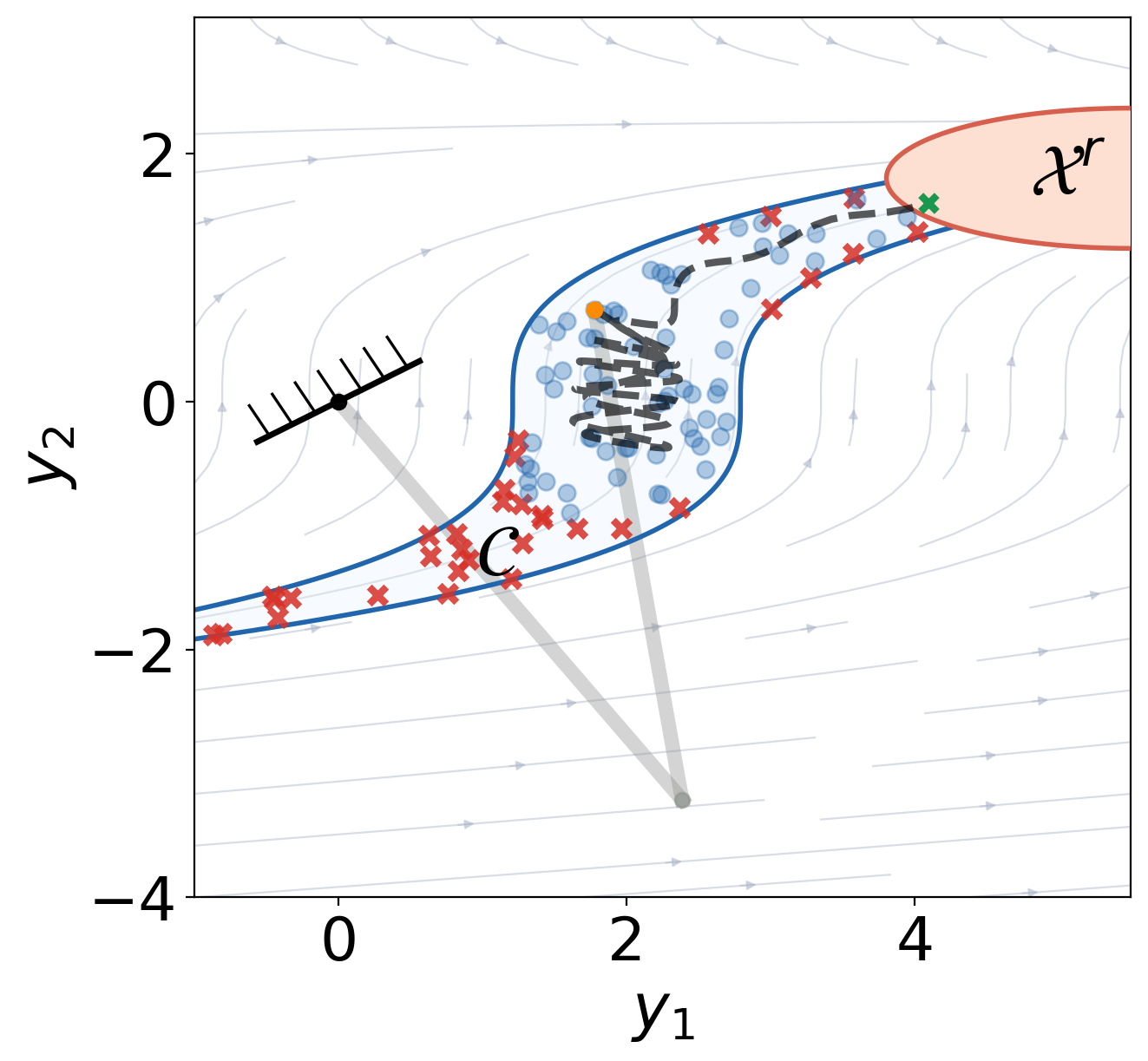}
        \caption{Unconstrained reach-avoid with  67/100.}
        \label{fig:unconstrained reach-avoid manipulator}
    \end{subfigure}
    \hfill
    \begin{subfigure}[t]{0.23\linewidth}
        \centering
        \includegraphics[width=\linewidth]{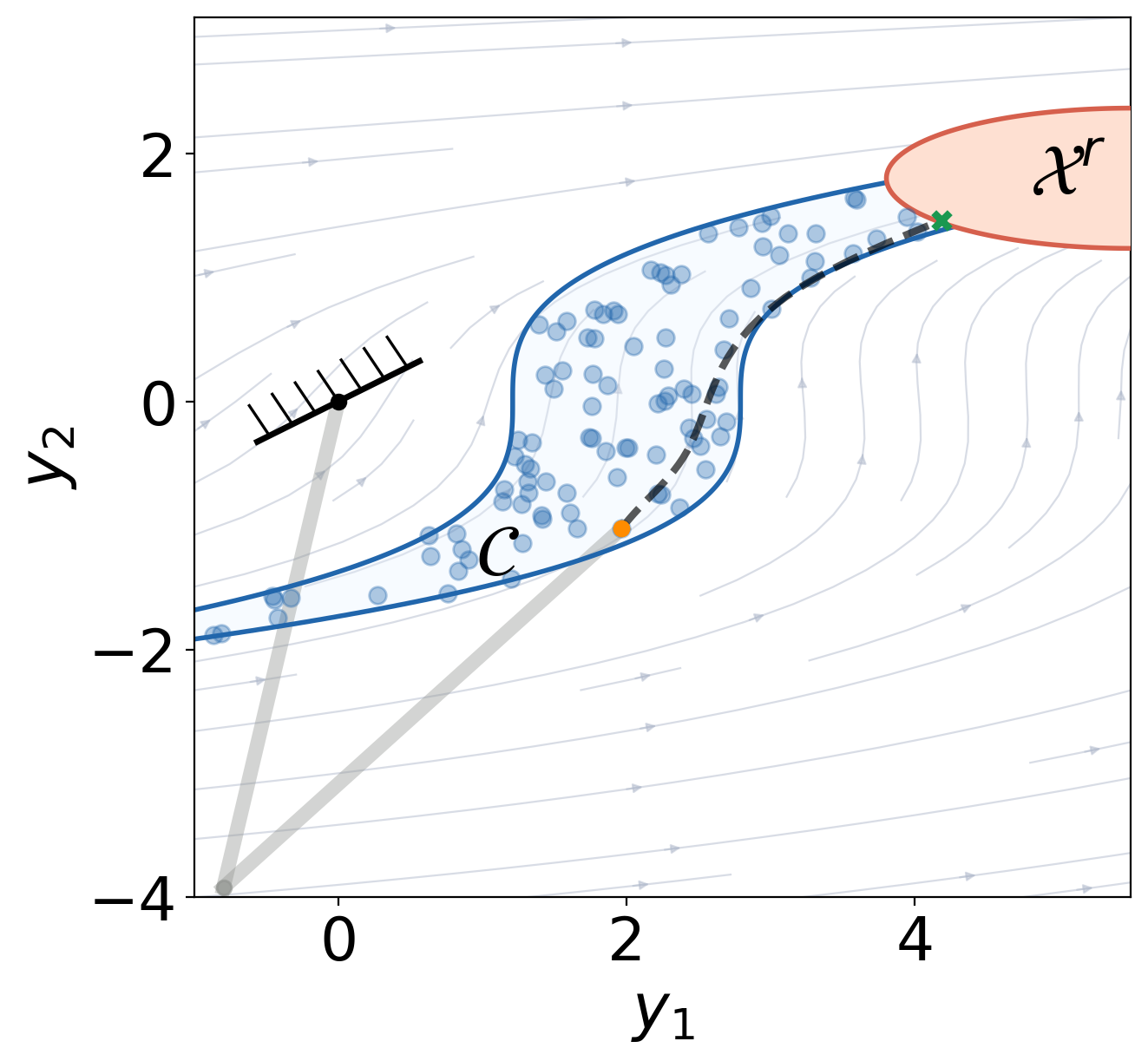}
        \caption{Our method with 100/100.}
        \label{fig:constrained reach-avoid mpc manipulator}
    \end{subfigure}
    \hfill
    \begin{subfigure}[t]{0.29\linewidth}
        \centering
        \includegraphics[width=\linewidth]{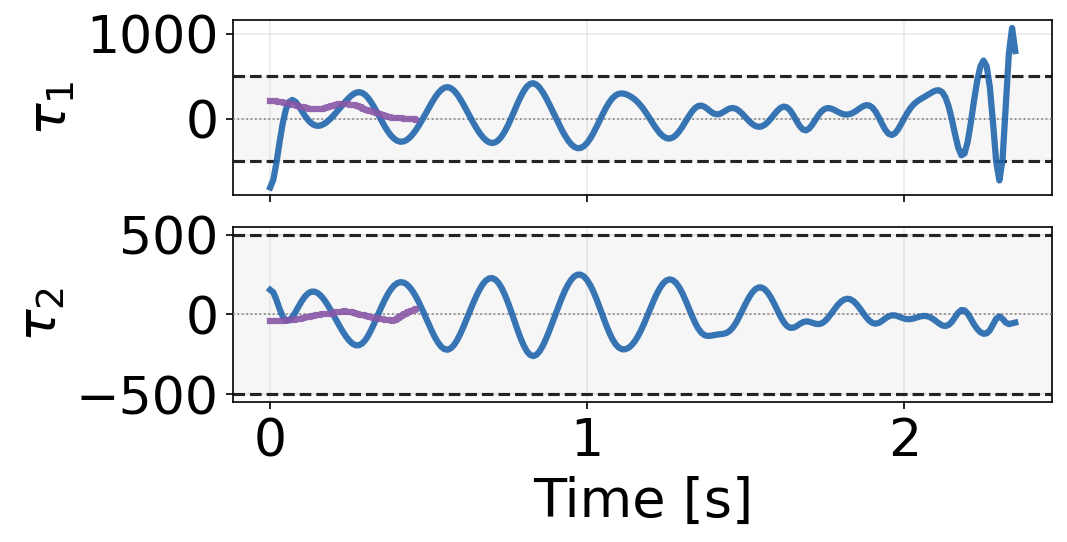}
        \caption{Control inputs comparison}
        \label{fig:control inputs comparison manipulator}
    \end{subfigure}    
    \caption{
    (\subref{fig:vanilla mpc manipulator}), (\subref{fig:unconstrained reach-avoid manipulator}) \& (\subref{fig:constrained reach-avoid mpc manipulator}) show controller success rates for $100$ random initial states. Trajectories starting from blue dots achieved the reach-avoid objective while the ones starting at the red crosses led to infeasibility. A feasible closed-loop trajectory (dashed line from orange dot to green cross) is included for visualization. (\subref{fig:control inputs comparison manipulator}) represents control inputs comparison between unconstrained reach-avoid controller (blue) and the proposed MPC controller (purple).
    }
    \label{fig:ex10}\vspace{-0.5cm}
\end{figure*}

\begin{example}
Consider the two-link planar robotic arm
\begin{equation}
    \dot{\bm{q}} =
    \left[\begin{array}{c}
    \dot{q}_1\\
    \dot{q}_2\\
    -\bm{M}^{-1}(\bm{C}\dot{\bm{q}}+\bm{G})
    \end{array}\right]+\left[\begin{array}{c}\bm{0}^{2\times2}\\\bm{M}^{-1}\end{array}\right]\left[\begin{array}{c}
    \tau_1\\
    \tau_2\end{array}\right].
\end{equation}
with $\bm{q} = [q_1,q_2,\dot{q}_1,\dot{q}_2]^\top \in \mathbb{R}^4$, and $[q_1, q_2]^\top$ as the joint angles, and $[\dot{q}_1, \dot{q}_2]^\top$ denoting the joint velocities. The inertia matrix $\bm{M}$ is defined as
$\bm{M} = \begin{bmatrix}
    M_{11} & M_{12}\\
    M_{21} & M_{22}
\end{bmatrix}$
where $M_{11} = I_1 + I_2 + m_2(l_1^2 + l_{c_2}^2 + 2l_1l_{c_2}\cos{(q_2)})+m_1l_{c_1}^2$, $M_{12} = M_{21} =  I_2 + m_2(l_{c_2}^2 + l_1l_{c_2}\cos{(q_2)})$, $M_{22} = m_2l_{c_2}^2 + I_2$.
\[
\bm{C} =
\begin{bmatrix}
- m_2 l_1 l_{c_2} \sin(q_2) \dot{q}_2 & - m_2 l_1 l_{c_2} \sin(q_2) (\dot{q}_1 + \dot{q}_2) \\
m_2 l_1 l_{c_2} \sin(q_2) \dot{q}_1 & 0
\end{bmatrix}
\] is the Coriolis matrix. The gravity vector is
\[
\bm{G} =
\begin{bmatrix}
(m_1 g l_{c_1} + m_2 g l_1) \cos(q_1) + m_2 g l_{c_2} \cos(q_1 + q_2) \\
m_2 g l_{c_2} \cos(q_1 + q_2)
\end{bmatrix}\]
 $\bm{\tau}\in\mathbb{R}^2$ is the joint torque as control inputs. Specific values of link parameters (mass, inertia, etc.) can be found in our code link.
 Given safe set $\mathcal{C}= \{ \bm{q} \in \mathbb{R}^4 \,|\, -(4(y_1(\bm{q})-2)-2y_2(\bm{q})^3)^2 + 0.8y_2(\bm{q})^3 + 10 >0 \}$ and target set $\mathcal{X}^r = \{ \bm{q} \in \mathbb{R}^4 \,|\, ((y_1(\bm{q})- 5.5)^2 / 1.2^2) + ((y_2(\bm{q}) - 1.8)^2 / 0.4^2) - 2 <0 \}$,
we aim to design a controller that enables the robotic arm to execute its motion within the specified workspace bounds, ultimately achieving predefined terminal configurations while respecting control input limits $ | \tau_{1,2} | \leq 500$. Both vanilla MPC and the proposed MPC framework with prediction horizon $N=20$ and sampling step $T=0.01$.
From Fig. \ref{fig:ex10}, as in our previous example, the reach-avoid MPC once again shows higher success rates in achieving the reach-avoid objective compared to both vanilla MPC and the unconstrained reach-avoid feedback controller, while having smoother trajectories compared to \ref{fig:unconstrained reach-avoid manipulator}. The reach-avoid MPC inputs remain bounded while the trajectories reach the target in a much shorter time, as seen from Fig. \ref{fig:constrained reach-avoid mpc manipulator}.
\end{example}

\section{Conclusion}

This paper proposed a sampled-data MPC framework for continuous control-affine systems that synthesizes controllers with rigorous reach-avoid guarantees. Under sufficiently fast sampling, the framework is shown to recursively admit feasible constrained control inputs that safely steer the system into the target set. 
Future work will focus on extending the framework to systems with external disturbances.

\bibliographystyle{ieeetr}
\bibliography{refs}

\clearpage
\section*{Appendix}

\subsection{Additional details for the proof of Lemma \ref{lemma:sampled-data system reach-avoid guarantees}}

\begin{align*}
        \left\| \bm{x}(t) - \bm{x}(t_k) \right\| \leq&
        \int_{t_k}^t 
        \left\|
        \bm{f}(\bm{x}(s)) + \sum_{j=1}^{m}\bm{g}_{j}(\bm{x}(s))u_{j}^c(t_k)
        \right\| \mathrm{d}s \\
        \leq&
        \int_{t_k}^t
        \bigg[
        \ell_{\bm{f}}
        \left\| \bm{x}(s) - \bm{x}(t_k) \right\| + \left\| \dot{\bm{x}}(t_k) \right\|\\
        & \qquad + \sum_{j=1}^m \ell_{\bm{g}_j} M_{u_j^c} \left\| \bm{x}(s) - \bm{x}(t_k) \right\|
        \bigg]
        \mathrm{d}s \\
        =& 
        \int_{t_k}^t
        \bigg[
        (\ell_{\bm{f}} + \sum_{j=1}^m \ell_{\bm{g}_j} M_{u_j^c}) \left\| \bm{x}(s) - \bm{x}(t_k) \right\| 
        \bigg] \mathrm{d}s \\
        & \qquad \qquad \qquad \qquad + (t-t_k) \left\| \dot{\bm{x}}(t_k) \right\|
\end{align*}
Let $\ell_{sys} = \ell_{\bm{f}} + \sum_{j=1}^m \ell_{\bm{g}_j} M_{u_j^c}$ and $C_k = \left\| \dot{\bm{x}}(t_k) \right\| $,
by applying the Grönwall–Bellman inequality, we obtain,
\begin{align*}
    \left\| \bm{x}(t) - \bm{x}(t_k) \right\| \leq&
    (t-t_k) C_k + 
    \int_{t_k}^t 
    (s-t_k) C_k 
    \ell_{sys} 
    e^{\ell_{sys} (t-s) } \mathrm{d}s \\
    &= (t-t_k) C_k + 
    C_k \ell_{sys}
    \int_{t_k}^t 
    (s-t_k) e^{\ell_{sys}(t-s)} \mathrm{d}s 
\end{align*}
Let $u=s-t_k, \ \mathrm{d}v = e^{-\ell_{sys}} \mathrm{d}s$, we have $\mathrm{d}u = \mathrm{d}s, \ v = -\frac{1}{\ell_{sys}} e^{-\ell_{sys}s}$, then
\begin{align*}
    \int_{t_k}^t 
    (s-t_k) e^{\ell_{sys}(t-s)} \mathrm{d}s  
    &= e^{\ell_{sys}t} 
    \left( \bigg[
    -\frac{s-t_k}{\ell_{sys}} e^{-\ell_{sys}s}
    \bigg] \Bigg|_{t_k}^t - 
    \int_{t_k}^t 
    \left[-\frac{1}{\ell_{sys}} e^{-\ell_{sys}s} \right] \mathrm{d}u
    \right) \\
    &=e^{\ell_{sys}t} 
    \left(
    -\frac{t-t_k}{\ell_{sys}} e^{-\ell_{sys}t} + 
    \frac{1}{\ell_{sys}}\left[ - \frac{1}{\ell_{sys} } e^{-\ell_{sys}s}\right] \Bigg |_{t_k}^t
    \right) \\
    &= e^{\ell_{sys}t} 
    \left(
    -\frac{t-t_k}{\ell_{sys}} e^{-\ell_{sys}t} - 
    \frac{1}{\ell_{sys}^2}\left[e^{-\ell_{sys}t} -e^{-\ell_{sys}t_k} \right] 
    \right) \\
    &=  - \frac{t-t_k}{\ell_{sys}} - \frac{1}{\ell_{sys}^2} + \frac{1}{\ell_{sys}^2} e^{\ell_{sys}(t-t_k)} \\
    &=\frac{1}{\ell_{sys}^2} \left( e^{\ell_{sys}(t-t_k)} -1 \right) - \frac{t-t_k}{\ell_{sys}}
\end{align*}
Thus,
\begin{align*}
    \left\| \bm{x}(t) - \bm{x}(t_k) \right\| \leq&
    (t-t_k) C_k + 
    C_k \ell_{sys} \left[ \frac{1}{\ell_{sys}^2} \left( e^{\ell_{sys}(t-t_k)} -1 \right) - \frac{t-t_k}{\ell_{sys}} \right] \\
    =& C_k \frac{e^{\ell_{sys}(t-t_k)} -1}{\ell_{sys}} \\
    =& \frac{e^{\ell_{sys}(t-t_k)} -1}{\ell_{sys}} \left\| \dot{\bm{x}}(t_k) \right\| \\
    =& \frac{e^{\left(\ell_{\bm{f}} + \sum_{j=1}^m \ell_{\bm{g}_j} M_{u_j^c}\right)(t-t_k)} -1}{\ell_{\bm{f}} + \sum_{j=1}^m \ell_{\bm{g}_j} M_{u_j^c}} \left\| \dot{\bm{x}}(t_k) \right\| \\
    =& \frac{e^{\left(\ell_{\bm{f}} + \sum_{j=1}^m \ell_{\bm{g}_j} M_{u_j^c}\right)(t-t_k)} -1}{\ell_{\bm{f}} + \sum_{j=1}^m \ell_{\bm{g}_j}M_{u_j^c}} \left\| \dot{\bm{x}}(t_k) \right\| \\
    \leq& \frac{e^{\left(\ell_{\bm{f}} + \sum_{j=1}^m \ell_{\bm{g}_j} M_{u_j^c}\right)T} -1}{\ell_{\bm{f}} + \sum_{j=1}^m \ell_{\bm{g}_j} M_{u_j^c}} \left\| \dot{\bm{x}}(t_k) \right\| 
\end{align*}

\end{document}